\documentclass[10pt]{extarticle}

\usepackage[english]{babel}
\usepackage{graphicx}
\usepackage{framed}
\usepackage[normalem]{ulem}
\usepackage{indentfirst}
\usepackage{amsmath,amsthm,amssymb,amsfonts}
\usepackage{mathtools} 
\usepackage[italicdiff]{physics} 
\usepackage[T1]{fontenc}
\usepackage{stmaryrd}
\usepackage{lmodern,mathrsfs}
\usepackage[inline,shortlabels]{enumitem}
\setlist{topsep=2pt,itemsep=2pt,parsep=0pt,partopsep=0pt}
\usepackage[table,dvipsnames]{xcolor}
\usepackage[utf8]{inputenc}
\usepackage{csquotes} 
\usepackage[a4paper,top=0.7in,bottom=0.7in,left=0.5in,right=0.5in]{geometry}
\usepackage[most]{tcolorbox}
\usepackage[bottom,multiple]{footmisc} 
\usepackage[backend=bibtex,style=numeric]{biblatex}
\usepackage{xurl} 
\usepackage[colorlinks,linkcolor=.,citecolor=blue,urlcolor=violet]{hyperref}
\usepackage[nameinlink]{cleveref} 

\renewcommand{\ref}[1]{%
  \textcolor{YellowGreen}{\normalfont\ref{#1}}%
}
\renewcommand{\eqref}[1]{%
  \textcolor{RoyalBlue}{\normalfont(\ref{#1})}%
}

\usepackage[showisoZ=false]{datetime2}

\DTMnewdatestyle{newdate}{

}

\newcommand*{\utcpm}[1]{
\ifboolexpe{test{\ifnumequal{#1}{0}}}{}{\ifnum#1<0\else+\fi{#1}}}
\DTMnewzonestyle{newzone}{
    
}
\AtBeginDocument{
    \DTMsetdatestyle{newdate}
    \DTMsetzonestyle{newzone}
}

\usepackage{fancyhdr}
\fancypagestyle{plain}{
\fancyhf{}

\fancyfoot[C]{\sffamily\thepage}}

\newcommand{\sub}{\subseteq}

\NewDocumentCommand{\exc}{ !g }{\backslash\IfValueT{#1}{\Bqty{#1}}}

\newcommand*{\setbuild}[2]{\left\{#1\,\middle|\,#2\right\}}

\docsvlist{
    deg,ord,
}

\newcommand*{\Z}{\opbraces{\mathbb{Z}}}

\newcommand*{\F}{\opbraces{\mathbb{F}}}
\newcommand*{\K}{\opbraces{\mathbb{K}}}

\newcommand*{\powerseries}[1]{\left\llbracket#1\right\rrbracket}
\newcommand*{\Os}{\mathcal{O}}
\newcommand*{\ftil}{\widetilde{f}}
\newcommand*{\quo}[1]{\textcolor{blue!80!black}{#1}}
\newcommand*{\rem}[1]{\textcolor{red!80!black}{#1}}

\makeatletter
\renewcommand{\impliedby}{\@ifstar{ \ensuremath{\Longleftarrow}}{ \ensuremath{\Leftarrow} }} 
\renewcommand{\implies}{\@ifstar{ \ensuremath{\Longrightarrow}}{ \ensuremath{\Rightarrow} }} 
\renewcommand{\iff}{\@ifstar{ \ensuremath{\Longleftrightarrow}}{ \ensuremath{\Leftrightarrow} }} 
\makeatother

\newtheoremstyle{mystyle}{}{}{}{}{\sffamily\bfseries}{.}{ }{}
\newtheoremstyle{cstyle}{}{}{}{}{\sffamily\bfseries}{.}{ }{\thmnote{#3}}
\makeatletter
\renewenvironment{proof}[1][\proofname] {\par\pushQED{\qed}{\normalfont\sffamily\bfseries\topsep6\p@\@plus6\p@\relax #1\@addpunct{.} }}{\popQED\endtrivlist\@endpefalse}
\makeatother
\theoremstyle{mystyle}{\newtheorem{definition}{Definition}[section]}
\theoremstyle{mystyle}{\newtheorem{proposition}[definition]{Proposition}}
\theoremstyle{mystyle}{\newtheorem{theorem}[definition]{Theorem}}
\theoremstyle{mystyle}{\newtheorem{lemma}[definition]{Lemma}}
\theoremstyle{mystyle}{\newtheorem{corollary}[definition]{Corollary}}
\theoremstyle{mystyle}{\newtheorem*{remark}{Remark}}
\theoremstyle{mystyle}{\newtheorem*{remarks}{Remarks}}
\theoremstyle{mystyle}{\newtheorem*{example}{Example}}
\theoremstyle{mystyle}{\newtheorem*{examples}{Examples}}
\theoremstyle{cstyle}{\newtheorem*{conjecture}{}}

\definecolor{tcol_DEF}{HTML}{E40125} 
\definecolor{tcol_PRP}{HTML}{EB8407} 
\definecolor{tcol_LEM}{HTML}{05C4D9} 
\definecolor{tcol_THM}{HTML}{1346E4} 
\definecolor{tcol_COR}{HTML}{7904C2} 
\definecolor{tcol_EXA}{HTML}{21340A} 
\definecolor{tcol_REM}{HTML}{18B640} 
\definecolor{tcol_PRF}{HTML}{5A76B2} 

\tcbset{
    tbox_DPLTC_style/.style={
        enhanced jigsaw,
        colback=#1!10, colframe=#1!80!black,
        boxrule=0pt,
        fonttitle=\sffamily\bfseries, coltitle=black,
        separator sign={}, label separator={},
        description font=\normalfont\sffamily,
        description delimiters={(}{)},
        attach title to upper, after title={.\ },
        frame hidden, borderline west={2pt}{0pt}{#1},
        sharp corners,
        top=2pt, bottom=2pt, left=5pt, right=5pt,
        beforeafter skip=10pt, breakable,
    },
    tbox_DPLTC_style*/.style={
        tbox_DPLTC_style=#1,
        interior hidden,
        top=-2pt, bottom=-2pt,
    },
}

\tcolorboxenvironment{definition}{tbox_DPLTC_style=tcol_DEF}
\tcolorboxenvironment{proposition}{tbox_DPLTC_style=tcol_PRP}
\tcolorboxenvironment{theorem}{tbox_DPLTC_style=tcol_THM}
\tcolorboxenvironment{lemma}{tbox_DPLTC_style=tcol_LEM}
\tcolorboxenvironment{corollary}{tbox_DPLTC_style=tcol_COR}
\tcolorboxenvironment{proof}{boxrule=0pt,boxsep=0pt,blanker,borderline west={2pt}{0pt}{CadetBlue!80!white},left=8pt,right=8pt,sharp corners,before skip=10pt,after skip=10pt,breakable}
\tcolorboxenvironment{remark}{tbox_DPLTC_style*=tcol_REM}
\tcolorboxenvironment{remarks}{tbox_DPLTC_style*=tcol_REM}
\tcolorboxenvironment{example}{tbox_DPLTC_style*=tcol_EXA}
\tcolorboxenvironment{examples}{tbox_DPLTC_style*=tcol_EXA}
\tcolorboxenvironment{conjecture}{boxrule=0pt,boxsep=0pt,colback={gray!10},left=8pt,right=8pt,enhanced jigsaw, borderline west={2pt}{0pt}{gray},sharp corners,before skip=10pt,after skip=10pt,breakable}

\AddToHook{env/proposition/begin}{\crefalias{definition}{proposition}}
\AddToHook{env/lemma/begin}{\crefalias{definition}{lemma}}
\AddToHook{env/theorem/begin}{\crefalias{definition}{theorem}}
\AddToHook{env/corollary/begin}{\crefalias{definition}{corollary}}

\crefname{definition}{definition}{definitions}
\crefformat{definition}{\normalfont\sffamily\bfseries\color{tcol_DEF!80} #2definition~#1#3}
\crefrangeformat{definition}{\normalfont\sffamily\bfseries\color{tcol_DEF} #3definitions~#1#4 {\normalfont\color{black}and} #5#2#6}
\Crefformat{definition}{\normalfont\sffamily\bfseries\color{tcol_DEF} #2Definition~#1#3}
\Crefrangeformat{definition}{\normalfont\sffamily\bfseries\color{tcol_DEF} #3Definitions~#1#4 {\normalfont\color{black}and} #5#2#6}

\crefname{proposition}{proposition}{propositions}
\crefformat{proposition}{\normalfont\sffamily\bfseries\color{tcol_PRP} #2proposition~#1#3}
\crefrangeformat{proposition}{\normalfont\sffamily\bfseries\color{tcol_PRP} #3propositions~#1#4 {\normalfont\color{black}and} #5#2#6}
\Crefformat{proposition}{\normalfont\sffamily\bfseries\color{tcol_PRP} #2Proposition~#1#3}
\Crefrangeformat{proposition}{\normalfont\sffamily\bfseries\color{tcol_PRP} #3Propositions~#1#4 {\normalfont\color{black}and} #5#2#6}

\crefname{lemma}{lemma}{lemmas}
\crefformat{lemma}{\normalfont\sffamily\bfseries\color{tcol_LEM} #2lemma~#1#3}
\crefrangeformat{lemma}{\normalfont\sffamily\bfseries\color{tcol_LEM} #3lemmas~#1#4 {\normalfont\color{black}and} #5#2#6}
\Crefformat{lemma}{\normalfont\sffamily\bfseries\color{tcol_LEM} #2Lemma~#1#3}
\Crefrangeformat{lemma}{\normalfont\sffamily\bfseries\color{tcol_LEM} #3Lemmas~#1#4 {\normalfont\color{black}and} #5#2#6}

\crefname{theorem}{theorem}{theorems}
\crefformat{theorem}{\normalfont\sffamily\bfseries\color{tcol_THM} #2theorem~#1#3}
\crefrangeformat{theorem}{\normalfont\sffamily\bfseries\color{tcol_THM} #3theorems~#1#4 {\normalfont\color{black}and} #5#2#6}
\Crefformat{theorem}{\normalfont\sffamily\bfseries\color{tcol_THM} #2Theorem~#1#3}
\Crefrangeformat{theorem}{\normalfont\sffamily\bfseries\color{tcol_THM} #3Theorems~#1#4 {\normalfont\color{black}and} #5#2#6}

\crefname{corollary}{corollary}{corollaries}
\crefformat{corollary}{\normalfont\sffamily\bfseries\color{tcol_COR} #2corollary~#1#3}
\crefrangeformat{corollary}{\normalfont\sffamily\bfseries\color{tcol_COR} #3corollaries~#1#4 {\normalfont\color{black}and} #5#2#6}
\Crefformat{corollary}{\normalfont\sffamily\bfseries\color{tcol_COR} #2Corollary~#1#3}
\Crefrangeformat{corollary}{\normalfont\sffamily\bfseries\color{tcol_COR} #3Corollaries~#1#4 {\normalfont\color{black}and} #5#2#6}

\usepackage[explicit]{titlesec}
\titleformat{\section}{\fontsize{16}{20}\sffamily\bfseries}{\thesection}{16pt}{#1}
\titleformat{\subsection}{\fontsize{12}{16}\sffamily\bfseries}{\thesubsection}{12pt}{#1}
\titleformat{\subsubsection}{\fontsize{10}{12}\sffamily\large\bfseries}{\thesubsubsection}{8pt}{#1}
\titlespacing*{\section}{0pt}{5pt}{5pt}
\titlespacing*{\subsection}{0pt}{5pt}{5pt}
\titlespacing*{\subsubsection}{0pt}{5pt}{5pt}

\DeclareMathAlphabet\mathbfcal{OMS}{cmsy}{b}{n}
\makeatletter
\g@addto@macro\normalsize{
    \setlength\abovedisplayskip{3pt}
    \setlength\belowdisplayskip{3pt}
    \setlength\abovedisplayshortskip{0pt}
    \setlength\belowdisplayshortskip{0pt}
}
\makeatother

\makeatletter
\renewcommand\maketitle{
    \begin{center}\sffamily
        {\huge\bfseries\@title}\\
            \vspace{6mm}
        {\Large\@author}\\
            \vspace{6mm}
        {\large\@date}
    \end{center}
}
\makeatother

\title{A Necessary and Sufficient Condition for Uniqueness of Euclidean Division}
\author{Senan Sekhon}
\date{\DTMToday} 

\begin{document}

\thispagestyle{plain}

\maketitle

\begin{abstract}
    A well-known result from the 1960s characterizes all Euclidean domains in which division is guaranteed to produce a unique quotient and remainder. As this relies on the historical (and more restrictive) definition of a Euclidean domain, the question of whether the result still holds under the modern definition was left open. In this paper, we prove the answer is afirmative.
\end{abstract}

\begin{tcolorbox}[enhanced,
    title={Contents},
    fonttitle=\Large\sffamily\bfseries\selectfont,
    fontupper=\sffamily,
    top=2mm,bottom=2mm,left=2mm,right=2mm,
    beforeafter skip=10mm,
    drop fuzzy shadow,breakable
]
    \makeatletter
    \@starttoc{toc}
    \makeatother
\end{tcolorbox}

\begin{tcolorbox}[enhanced,
    frame hidden, interior hidden,
    title={Conventions},
    coltitle=black,
    fonttitle=\large\sffamily\bfseries\selectfont,
    top=2mm,bottom=2mm,left=2mm,right=2mm,
    beforeafter skip=10mm
]
    \begin{itemize}[label={},leftmargin=0pt]
        \item $\Z_0^+$ denotes the set $\{0,1,2,3,\ldots\}$ of non-negative integers (including $0$).
        \item $\K$ denotes an arbitrary field.
        \item $R$ denotes an integral domain (a commutative ring with a multiplicative identity and no nonzero zero divisors).
        \item $f$ denotes a Euclidean function on $R$.
    \end{itemize}
\end{tcolorbox}

\section{Background}

In a Euclidean domain, the quotient and remainder of Euclidean division are generally not unique. A well-known result, proved independently by Rhai (1962) \cite{rhai} and Jodeit (1967) \cite{jodeit}, fully characterizes the cases where they \emph{are} unique: If $R$ is a Euclidean domain in which the quotient and remainder of Euclidean division are always unique, then either $R$ is a field or $R\cong\K[x]$ for some field $\K$.\\

However, both Rhai and Jodeit used the original definition of a Euclidean domain (\Cref{def:strongly_euclidean}), which was standard at the time but more restrictive than the modern definition (\Cref{def:euclidean_domain}), involving an extra assumption for the Euclidean function. In other words, the old and new definitions agree on what a Euclidean \emph{domain} is, but they disagree on what qualifies as a Euclidean \emph{function}. A result of Rogers (1971) \cite{rogers} states that one can always enforce this assumption by constructing a new ``refined'' Euclidean function. \\

Since uniqueness depends on the Euclidean \emph{function} (not just the Euclidean \emph{domain} itself), this left a gap in the literature: Jodeit and Rhai's result holds assuming uniqueness under a \emph{refined} Euclidean function, as per Rogers' construction. Is uniqueness under \emph{some} Euclidean function (without any modification) enough to conclude that $R$ is a field or $R\cong\K[x]$? This paper closes that gap.

\section{Preliminary definitions}

\begin{definition}\label{def:euclidean_domain}
    A \textbf{Euclidean domain} (or \emph{Euclidean ring}) is an integral domain $R$, together with a function $f:R\exc{0}\to\Z_0^+$, such that:
    \begin{equation*}
        \text{For all } a,b\in R, b\ne 0, \text{ there exist } q,r\in R \text{ such that } a=qb+r \text{ and } (r=0 \text{ or } f(r)<f(b))
    \end{equation*}
    Any function $f$ satisfying this condition is known as a \textbf{Euclidean function} on $R$.
\end{definition}
\begin{remarks}\leavevmode
    \begin{enumerate}
        \item The Euclidean function $f$ is \emph{not} required to be defined at $0$, and defining $f(0)$ is superfluous as we already explicitly allow $r=0$ in the definition.
        \item By convention, the codomain of $f$ is taken to be $\Z_0^+$, but it may equivalently be any non-empty subset of $\Z$ that is bounded below (this is to ensure, crucially, that it is well-ordered).
    \end{enumerate}
\end{remarks}

\begin{examples}\leavevmode
    \begin{enumerate}
        \item Every field is Euclidean with respect to every function (since we can always set $q=ab^{-1}$ and $r=0$).
        \item $\Z$ is Euclidean with respect to the absolute value $f(n)=\abs{n}$.
    \end{enumerate}
\end{examples}

\begin{example}[Polynomial ring]
    Suppose $\K$ is a field. The ring $\K[x]$ of formal polynomials in one variable with coefficients in $\K$ is Euclidean with respect to the \emph{degree} function $\deg:\K[x]\exc{0}\to\Z_0^+$. Explicitly:
    \begin{align*}
        \K[x] = \setbuild{\sum_{n=0}^N a_nx^n}{N\in\Z_0^+, a_n\in\K} &&
        \deg(\sum_{n=0}^N a_nx^n) = \max\setbuild{n\in\Z_0^+}{a_n\ne 0}
    \end{align*}
    It is a common convention to set $\deg(0)=-\infty$, but we will not need it here.
\end{example}

\begin{example}[Power series ring]
    Suppose $\K$ is a field. The ring $\K\powerseries{x}$ of formal power series in one variable with coefficients in $\K$ is Euclidean with respect to the \emph{order} function $\ord:\K\powerseries{x}\exc{0}\to\Z_0^+$. Explicitly:
    \begin{align*}
        \K\powerseries{x} = \setbuild{\sum_{n=0}^\infty a_nx^n}{a_n\in\K} &&
        \ord(\sum_{n=0}^\infty a_nx^n) = \min\setbuild{n\in\Z_0^+}{a_n\ne 0}
    \end{align*}
    It is a common convention to set $\ord(0)=+\infty$, but we will not need it here.
\end{example}

\begin{example}[Quadratic integer rings]
    For each $d\in\Z$, the \emph{quadratic integer ring} $\Os_d$ and its \emph{norm} $N:\Os_d\to\Z_0^+$ are given by:
    \begin{align*}
        \Os_d =
        \begin{cases}
            \Z[\frac{1+\sqrt{d}}{2}] & d\equiv 1\pmod{4} \\[5pt]
            \Z[\sqrt{d}] & \text{otherwise}
        \end{cases}
        &&
        N\qty(a+b\sqrt{d}) = a^2-b^2d
    \end{align*}
    The question ``For which values of $d$ is $N$ Euclidean on $\Os_d$?'' led to a century-long investigation by dozens of mathematicians, ending in 1952. The result is:
    \begin{center}
        $N$ is Euclidean on $\Os_d$ if and only if $d=-11,-7,-3,-2,-1,2,3,5,6,7,11,13,17,19,21,29,33,37,41,57,73$.
    \end{center}
    See \cite[\S4]{lemmermeyer} for a compilation of the results that led to this classification.
\end{example}

In what follows, $R$ always denotes an integral domain and $f$ always denotes a Euclidean function on $R$, unless otherwise specified. For brevity, we will merely write ``$f$ is Euclidean'', where the underlying integral domain is assumed to be $R$.\\

\Cref{def:euclidean_domain} is the modern, widely accepted definition of a Euclidean domain. All subsequent terminology is non-standard and solely to improve readability.

\begin{definition}\label{def:strongly_euclidean}
    $f$ is \textbf{strongly Euclidean} on $R$ if $f$ is Euclidean on $R$, and for all $a,b\in R\exc{0}$, we have $f(a)\le f(ab)$.
\end{definition}

\begin{examples}\leavevmode
    \begin{enumerate}
        \item The absolute value is strongly Euclidean on $\Z$.
        \item The degree is strongly Euclidean on $\K[x]$.
        \item The order is strongly Euclidean on $\K\powerseries{x}$.
        \item The norm $N$ is strongly Euclidean on $\Os_d$ if and only if it is Euclidean on $\Os_d$.
        \item If $R$ is a field, then $f$ is strongly Euclidean on $R$ if and only if it is constant (otherwise, we would have $f(a)>f(c)$ for some $a,c\in R\exc{0}$, and so $f(a)>f(a(a^{-1}c)$).
    \end{enumerate}
\end{examples}

The following lemma characterizes exactly when the inequality in \Cref{def:strongly_euclidean} is an equality.

\begin{lemma}\label{strongly_euclidean_equality_iff_unit}
    Suppose $f$ is strongly Euclidean on $R$. Then for all $a,b\in R\exc{0}$, we have $f(a)=f(ab)$ if and only if $b\in R^\times$.
\end{lemma}
\begin{proof}
    \begin{itemize}
        \item[($\impliedby$)] Suppose $b\in R^\times$. Then $bc=1$ for some $c\in R$. This yields $f(ab)\le f(abc)=f(a)$, and so $f(a)=f(ab)$.
        \item[($\implies$)] Suppose $f(a)=f(ab)$ but $b\notin R^\times$. Since $f$ is a Euclidean function on $R$, there exist $q,r\in R$ such that:
        \begin{equation*}
            a = q(ab)+r \qq{ and } (r=0 \text{ or } f(r)<f(ab))
        \end{equation*}
        This can be rewritten as:
        \begin{equation*}
            a(1-bq)=r \qq{ and } (r=0 \text{ or } f(r)<f(a))
        \end{equation*}
        Since $b\notin R^\times$, we have $1-bq\ne 0$, and since $R$ is an integral domain, we have $a(1-bq)\ne 0$. Thus $r\ne 0$, and so $f(a(1-bq))<f(a)$, a contradiction as $f$ is strongly Euclidean. Thus $b\in R^\times$. \qedhere
    \end{itemize}
\end{proof}
\begin{remark}
    This does NOT mean $f(x)=f(y)\implies x$ and $y$ are unit multiples of each other. It means that \emph{if} $y$ is a multiple of $x$ (or vice versa) and $f(x)=f(y)$, then $x$ and $y$ are unit multiples of each other.
\end{remark}

\begin{example}
    In $\Z[i]$, we have $N(2+i)=N(2-i)=5$, but $2+i$ and $2-i$ are not multiples of each other in $\Z[i]$.
\end{example}

\begin{lemma}\label{strongly_euclidean_minimized_by_units}
    Suppose $f$ is strongly Euclidean on $R$, and define $m=\min_{a\in R\exc{0}} f(a)$. Then $f(b)=m$ if and only if $b\in R^\times$.\\
    In other words, the minimum value of a strongly Euclidean function is attained precisely by the units.
\end{lemma}

This follows directly by substituting $a=1$ into \Cref{def:strongly_euclidean} and \Cref{strongly_euclidean_equality_iff_unit}.\\

\Cref{def:strongly_euclidean} was the original definition of a Euclidean domain (an integral domain with a \emph{strongly} Euclidean function), and is used in \cite{herstein} and \cite{lovett}.\\

Rogers proved in 1971 \cite{rogers} that every Euclidean function can be ``refined'' into a strongly Euclidean function, via the following construction:

\begin{definition}\label{def:refinement}
    Suppose $f$ is Euclidean on $R$. The \textbf{refinement} of $f$, denoted by $\ftil$, is given by:
    \begin{equation}\label{eq:refinement}
        \ftil(a) = \min_{b\in R\exc\{0\}} f(ab)
    \end{equation}
\end{definition}

\begin{proposition}\label{euclidean_function_refinement}
    Suppose $f$ is Euclidean on $R$. Then $\ftil$ is strongly Euclidean on $R$.
\end{proposition}
\begin{proof}
    Note that $\ftil$ is well-defined, since for every $a\in R\exc\{0\}$, the set $\setbuild{f(ab)}{b\in R\exc\{0\}}$ is a non-empty subset of $\Z_0^+$, so it has a minimum. Also, for all $a\in R\exc\{0\}$, we have $\ftil(a)\le f(a\cdot 1)=f(a)$.\\
    
    We first show that $\ftil$ is Euclidean on $R$. Suppose $a,b\in R$, $b\ne 0$. By \eqref{eq:refinement}, there exists $c\in R\exc\{0\}$ such that $\ftil(b)=f(bc)$. Since $f$ is Euclidean on $R$, there exist $q,r\in R$ such that $a=q(bc)+r$ and ($r=0$ or $f(r)<f(bc)$). This yields $a=(cq)b+r$. We want to show that $r=0$ or $\ftil(r)<\ftil(bc)$. If $r=0$, we are done. If not, we have:
    \begin{equation*}
        \ftil(r) \le f(r) < f(bc) = \ftil(b)
    \end{equation*}
    Thus $\ftil$ is Euclidean on $R$.\\

    Now suppose $a,b\in R\exc\{0\}$. By \eqref{eq:refinement}, there exists $c\in R\exc{0}$ such that $\ftil(ab)=f(abc)$. Since $abc$ is also a nonzero multiple of $a$, we have:
    \begin{equation*}
        \ftil(a) \le f(abc) = \ftil(ab)
    \end{equation*}
    Thus $\ftil(a)\le\ftil(ab)$, and so $\ftil$ is strongly Euclidean on $R$.
\end{proof}

\begin{corollary}\label{strongly_euclidean_preserves_refinement}
    $f$ is strongly Euclidean on $R$ if and only if $\ftil=f$.
\end{corollary}
\begin{proof}
    \begin{itemize}
        \item[($\impliedby$)] By \Cref{euclidean_function_refinement}, $\ftil$ is strongly Euclidean, so if $\ftil=f$, then $f$ is strongly Euclidean.
        \item[($\implies$)] Suppose $a\in R\exc{0}$. As we have shown in \Cref{euclidean_function_refinement}, $\ftil(a)\le f(a)$. Since $f$ is strongly Euclidean, we have $f(ab)\ge f(a)$ for all $b\in R\exc{0}$, and so $\ftil(a)\ge f(a)$. Thus $\ftil(a)=f(a)$. \qedhere
    \end{itemize}
\end{proof}

\begin{definition}\label{def:ultra_euclidean}
    $f$ is \textbf{ultra-Euclidean} on $R$ if $f$ is Euclidean on $R$, and for all $a,b\in R\exc{0}$ such that $a+b\ne 0$, we have $f(a+b)\le\max\{f(a),f(b)\}$.
\end{definition}
\begin{remark}
    The name ``ultra-Euclidean'' comes from the similarity to the triangle inequality in \emph{ultrametric} spaces: $d(x,z)\le\max\{d(x,y),d(y,z)\}$.
\end{remark}

\begin{examples}\leavevmode
    \begin{enumerate}
        \item The degree is ultra-Euclidean on $\K[x]$.
        \item The order is \emph{not} ultra-Euclidean on $\K\powerseries{x}$, since $\ord(1+(x-1))=\ord(x)=1>0=\max\{\ord(1),\ord(x-1)\}$.
        \item The absolute value is \emph{not} ultra-Euclidean on $\Z$, since $\abs{1+1}=2>1=\max\{\abs{1},\abs{1}\}=1$.
    \end{enumerate}
\end{examples}

\begin{definition}
    $f$ is \textbf{uniquely Euclidean} on $R$ if for all $a,b\in R$, $b\ne 0$, there exist \emph{unique} $q,r\in R$ such that $a=qb+r$ and ($r=0$ or $f(r)<f(b)$).
\end{definition}

\begin{examples}\leavevmode
    \begin{enumerate}
        \item The degree is uniquely Euclidean on $\K[x]$.
        \item The order is \emph{not} uniquely Euclidean on $\K\powerseries{x}$, since $1=0\cdot x+1=1\cdot x+(1-x)$, and $\ord(1)=\ord(1-x)<\ord(x)$. In fact, this works with $1=a\cdot x+(1-ax)$ for \emph{any} $a\in\K$. Thus, if $\K$ is infinite, there are always \emph{infinitely} many Euclidean divisions of $1$ by $x$.
        \item The absolute value is \emph{not} uniquely Euclidean on $\Z$, since $1=0\cdot 2+1=1\cdot 2+(-1)$ and $\abs{1}=\abs{-1}<\abs{2}$. In fact, for all $a,b\in\Z$, $b\ne 0$, if $b$ does not divide $a$, there will always be two Euclidean divisions of $a$ by $b$, one with a positive remainder and one with a negative remainder.
    \end{enumerate}
\end{examples}

\section{The original result}

We now state Jodeit and Rhai's main result in our terminology:

\begin{theorem}\label{jodeit_rhai}
    Suppose $f$ is strongly Euclidean and uniquely Euclidean on $R$. Then either $R$ is a field or $R\cong\K[x]$ for some field $\K$.
\end{theorem}

See \cite{jodeit} and \cite{rhai} for a proof.

\begin{remark}
    The conclusion of the theorem does not make any reference to $f$, so it applies to all integral domains that have \emph{at least one} function that is both strongly Euclidean and uniquely Euclidean.
\end{remark}

We will use the following terminology: A \textbf{candidate division} of $x$ by $y$ (where $x,y\in R$ and $y\ne 0$) is any expression of the form $x=\quo{q}y+\rem{r}$, where $q,r\in R$. It is \textbf{valid} if $r=0$ or $f(r)<f(y)$.\\

As such, saying that $f$ is Euclidean (resp. uniquely Euclidean) simply amounts to saying that for all $x,y\in R$, $y\ne 0$, there is at least one (resp. exactly one) valid candidate division of $x$ by $y$.\\

In their original proofs, both Jodeit \cite{jodeit} and Rhai \cite{rhai} prove the following result as a lemma, which we restate and prove using our terminology.

\begin{proposition}\label{uniquely_euclidean_iff_ultra_euclidean}
    Suppose $f$ is strongly Euclidean on $R$. Then $f$ is uniquely Euclidean if and only if it is ultra-Euclidean.
\end{proposition}
\begin{proof}
    \begin{itemize}
        \item[($\implies$)] Suppose $f$ is not ultra-Euclidean. Then there exist $a,b\in R\exc{0}$ such that $a+b\ne 0$ and $f(a+b)>\max\{f(a),f(b)\}$. Note that the following are candidate divisions of $a$ by $a+b$:
        \begin{align}
            a &= \quo{0} \cdot (a+b) + \rem{a}
            \label{eq:uniquely_euclidean_iff_ultra_euclidean:1} \\
            a &= \quo{1} \cdot (a+b) + \rem{(-b)}
            \label{eq:uniquely_euclidean_iff_ultra_euclidean:2}
        \end{align}
        These are distinct as $\quo{0}\ne\quo{1}$. By assumption, we have $f(\rem{a})<f(a+b)$, so \eqref{eq:uniquely_euclidean_iff_ultra_euclidean:1} is valid. Also, since $f$ is strongly Euclidean and $-1\in R^\times$, by \Cref{strongly_euclidean_equality_iff_unit}, we have:
        \begin{equation*}
            f(\rem{-b}) = f((-1)b) = f(b) < f(a+b)
        \end{equation*}
        Thus \eqref{eq:uniquely_euclidean_iff_ultra_euclidean:2} is also valid, and so $f$ is not uniquely Euclidean.
        \item[($\impliedby$)] Suppose $f$ is not uniquely Euclidean. Then there exist $a,b\in R$, $b\ne 0$ such that there are at least two distinct candidate divisions of $a$ by $b$. In other words, there exist $q,r,s,t\in R$ such that:
        \begin{equation*}
            a = qb+r = sb+t
            \qq{and}
            q\ne s
            \qq{and}
            (r=0 \text{ or } f(r)<f(b))
            \qq{and}
            (t=0 \text{ or } f(t)<f(b))
        \end{equation*}
        \begin{itemize}[label=\textbullet]
            \item If $r=0$, then $a=qb=sb+t$, so $t=(q-s)b$. Since $f$ is strongly Euclidean and $q-s\ne 0$, we have $f(t)=f((q-s)b)\ge f(b)$.
            \item If $t=0$, then $a=qb+r=sb$, so $r=(s-q)b$. Since $f$ is strongly Euclidean and $s-q\ne 0$, we have $f(r)=f((s-q)b)\ge f(b)$.
            \item If $r,t\ne 0$, then $f(b)\le f((q-s)b)=f(t-r)\le\max\{f(t),f(-r)\}=\max\{f(t),f(r)\}$.
        \end{itemize}
        In any case, we get a contradiction, as either $f(r)\ge f(b)$ or $f(t)\ge f(b)$. \qedhere
    \end{itemize}
\end{proof}
\begin{remark}
    Both directions of this proof use the assumption that $f$ is strongly Euclidean. The result is not true without this assumption, as the next example shows.
\end{remark}

\begin{example}[Ultra-Euclidean $\not\implies$ strongly Euclidean]
    Suppose $R=\F_4=\{0,1,\alpha,\beta\}$, the field with four elements. Define $f:R\exc{0}\to\Z_0^+$ by $f(1)=0$, $f(\alpha)=f(\beta)=1$. Since $R$ is a field, $f$ is Euclidean.\\
    
    The only pairs $(a,b)$ of nonzero elements of $R$ whose sum is nonzero are $(1,\alpha)$, $(1,\beta)$ and $(\alpha,\beta)$. In all cases, we have $f(a+b)\le\max\{f(a),f(b)\}$, since the right side is $1$ and the left side is either $0$ or $1$. Thus $f$ is ultra-Euclidean.\\
    
    However, $f$ is not strongly Euclidean as $f(\alpha)=1>0=f(1)=f(\alpha\beta)$. It is also not uniquely Euclidean as $1=\quo{0}\alpha+\rem{1}$ and $1=\quo{1}\alpha+\rem{\beta}$ are two valid candidate divisions of $1$ by $\alpha$.
\end{example}

\section{The improved result}

The key ingredient to improving Jodeit and Rhai's result (\Cref{jodeit_rhai}) is the following:

\begin{theorem}\label{uniquely_euclidean_implies_strongly_euclidean}
    If $f$ is uniquely Euclidean on $R$, then $f$ is strongly Euclidean on $R$.
\end{theorem}
\begin{proof}
    Suppose $a,b\in R\exc{0}$. We want to show that $f(a)\le f(ab)$. Note that the following is a candidate division of $a$ by $ab$:
    \begin{equation}\label{eq:uniquely_euclidean_implies_strongly_euclidean:1}
        a = \quo{0}\cdot (ab) + \rem{a}
    \end{equation}
    Since $a\ne 0$, this is valid if and only if $f(a)<f(ab)$.
    \begin{enumerate}
        \item If \eqref{eq:uniquely_euclidean_implies_strongly_euclidean:1} is valid, then $f(a)<f(ab)$, and we are done.
        \item If \eqref{eq:uniquely_euclidean_implies_strongly_euclidean:1} is not valid, then $f(a)\ge f(ab)$. We will show that $f(a)=f(ab)$. Suppose $a=\quo{q}(ab)+\rem{r}$ is the unique valid division of $a$ by $ab$, i.e. $\quo{q},\rem{r}\in R$ and ($\rem{r}=0$ or $f(\rem{r})<f(ab)$). Then $r=a(1-qb)$.
        \begin{enumerate}[label=(\roman{*})]
            \item If $r=0$, then $1-qb=0$, so $qb=1$. Note that the following are candidate divisions of $ab$ by $a$:
            \begin{align}
                ab &= \quo{b}\cdot a + \rem{0}
                \label{eq:uniquely_euclidean_implies_strongly_euclidean:2A} \\
                ab &= \quo{0}\cdot a + \rem{ab}
                \label{eq:uniquely_euclidean_implies_strongly_euclidean:2B}
            \end{align}
            Note that \eqref{eq:uniquely_euclidean_implies_strongly_euclidean:2A} is always valid (since it has remainder $\rem{0}$), and \eqref{eq:uniquely_euclidean_implies_strongly_euclidean:2B} is valid if and only if $f(ab)<f(a)$. If so, these candidates must be equal, and so $\quo{b}=\quo{0}$, a contradiction. Thus \eqref{eq:uniquely_euclidean_implies_strongly_euclidean:2B} is invalid, so $f(ab)\ge f(a)$, and so $f(ab)=f(a)$.
            \item If $r\ne 0$, then $f(r)<f(ab)$. Note that the following are candidate divisions of $r$ by $a$:
            \begin{align}
                r &= \quo{(1-qb)}\cdot a + \rem{0}
                \label{eq:uniquely_euclidean_implies_strongly_euclidean:3A} \\
                r &= \quo{0}\cdot a + \rem{r}
                \label{eq:uniquely_euclidean_implies_strongly_euclidean:3B}
            \end{align}
            Note that \eqref{eq:uniquely_euclidean_implies_strongly_euclidean:3A} is always valid (since it has remainder $\rem{0}$), and \eqref{eq:uniquely_euclidean_implies_strongly_euclidean:3B} is valid if and only if $f(r)<f(a)$. If so, these candidates must be equal, and so $\rem{r}=\rem{0}$, a contradiction. Thus \eqref{eq:uniquely_euclidean_implies_strongly_euclidean:3B} is invalid, so $f(r)\ge f(a)$, and so $f(a)\le f(r)<f(ab)\le f(a)$, a contradiction.
        \end{enumerate}
        Thus $f(a)\le f(ab)$. \qedhere
    \end{enumerate}
\end{proof}

This result immediately yields a necessary and sufficient condition for a function to be uniquely Euclidean.

\begin{corollary}\label{uniquely_euclidean_iff_strongly_euclidean_and_ultra_euclidean}
    $f$ is uniquely Euclidean if and only if $f$ is strongly Euclidean and ultra-Euclidean.
\end{corollary}

This follows by combining \Cref{uniquely_euclidean_iff_ultra_euclidean} and \Cref{uniquely_euclidean_implies_strongly_euclidean}.

\begin{theorem}
    If $f$ is uniquely Euclidean on $R$, so is $\ftil$.
\end{theorem}
\begin{proof}
    By \Cref{uniquely_euclidean_implies_strongly_euclidean}, $f$ is strongly Euclidean, so by \Cref{strongly_euclidean_preserves_refinement}, $\ftil=f$, and so $\ftil$ is uniquely Euclidean.
\end{proof}

We are now ready to state and prove our improvement of Jodeit and Rhai's result (\Cref{jodeit_rhai}). We use some ideas from their original proofs \cite{jodeit} and \cite{rhai}, reformulated in our terminology. The crucial enhancement is that due to \Cref{uniquely_euclidean_implies_strongly_euclidean}, we can drop the assumption that $f$ is strongly Euclidean.

\begin{theorem}\label{uniquely_euclidean_implies_K_or_Kx}
    Suppose $f$ is uniquely Euclidean on $R$. Then either $R$ is a field or $R\cong\K[x]$ for some field $\K$.
\end{theorem}
\begin{proof}
    By \Cref{uniquely_euclidean_implies_strongly_euclidean}, $f$ is strongly Euclidean, so by \Cref{strongly_euclidean_minimized_by_units}, $f(1)$ is the minimum value of $f$ and is attained precisely by the units.\\

    Define $K=R^\times\cup\{0\}$. We first show that $K$ is a field. Clearly $K$ is closed under multiplication, and $K\exc{0}=R^\times$, so every nonzero element of $K$ is invertible. Suppose $u,v\in K$. If $u=0$, $v=0$ or $u+v=0$, then clearly $u+v\in K$. Suppose $u,v,u+v\ne 0$. Then $u,v\in R^\times$, so $f$ attains its minimum at $u$ and $v$. By \Cref{uniquely_euclidean_iff_strongly_euclidean_and_ultra_euclidean}, $f$ is ultra-Euclidean, so $f(u+v)\le\max\{f(u),f(v)\}$. Thus $f$ also attains its minimum at $u+v$, and so $u+v\in R^\times\sub K$. Thus $K$ is also closed under addition, and so it is a field.\\

    If $K=R$, then $R$ is a field and we are done. Suppose $K\ne R$. Then there exists $x\in R$ such that $f(x)=\min_{y\in R\exc K} f(y)$.\\

    Suppose $a\in R$. We will show that $a$ can be expressed uniquely as a polynomial in $x$ with coefficients in $K$. By assumption, there exist unique $q,r\in R$ such that $a=qx+r$ and ($r=0$ or $f(r)<f(x)$). In other words, $r=0$ or $r\in R^\times$, and so $r\in K$. We first show that $f(qx)\le f(a)$.
    \begin{itemize}
        \item If $r=0$, then $a=qx$, and so $f(qx)=f(a)$.
        \item If $r\ne 0$, then $a\ne 0$ (otherwise $a=qx+r=0\cdot x+0$ would be two valid candidate divisions of $a$ by $x$). By \Cref{uniquely_euclidean_iff_strongly_euclidean_and_ultra_euclidean}, $f$ is ultra-Euclidean, so $f(qx)=f(a-r)\le\max\{f(a),f(-r)\}$. Since $f$ is strongly Euclidean, we have $f(-r)=f(r)<f(x)\le f(qx)$, so $\max\{f(a),f(-r)\}=f(a)$, and so $f(qx)\le f(a)$.
    \end{itemize}
    Since $x$ is not a unit, by \Cref{strongly_euclidean_equality_iff_unit}, we have $f(q)<f(qx)$, and so $f(q)<f(a)$.\\

    If $f(q)<f(x)$, then $q\in K$, so $a=qx+r$ is a polynomial in $x$ with coefficients in $K$. If $f(q)\ge f(x)$, we can repeat this process with $q$ in place of $a$. This yields unique $q_1,r_1\in R$ such that $q=q_1x+r_1$ and $r_1\in K$, as well as $f(q_1)<f(q)$. Repeating this process further, we get a sequence $f(a)>f(q_0)>f(q_1)>\cdots$. This is a strictly decreasing sequence in $\Z_0^+$, so it must terminate, say at $q_{n-1}$. This yields $a=q_{n-1}x^n+r_{n-1}x^{n-1}+\cdots+r_1x+r_0$, where $r_0,r_1,\ldots,r_{n-1}\in K$ (and since $f(q_{n-1})<f(x)$, we also have $q_{n-1}\in K$). Thus every $a\in R$ can be expressed uniquely as a polynomial in $x$ with coefficients in $K$, and so $R\cong K[x]$.
\end{proof}

Since a Euclidean function on a field is strongly Euclidean if and only if it is constant, this immediately characterizes all strongly Euclidean (and \emph{a fortiori} all uniquely Euclidean) functions on a field: They are the constant functions $f:\K\exc{0}\to\Z_0^+$.\\

As for $\K[x]$, the situation is just as nice. \cite[Theorem 2.2]{rhai} states that every uniquely Euclidean function $f$ on $\K[x]$ is of the form $\phi\circ\deg$, where $\phi:\Z_0^+\to\Z_0^+$ is strictly increasing\footnote{More specifically, it states this result for functions on $\K[x]$ that are both strongly Euclidean \emph{and} uniquely Euclidean, but we can use \Cref{uniquely_euclidean_implies_strongly_euclidean} to drop the `strongly Euclidean' assumption.}.\\

Combining these results with \Cref{uniquely_euclidean_implies_K_or_Kx}, we obtain a full characterization of \emph{all} uniquely Euclidean functions:

\begin{corollary}\label{uniquely_euclidean_function_characterization}
    Suppose $f$ is uniquely Euclidean on $R$. Then exactly one of the following holds:
    \begin{enumerate}
        \item $R$ is a field and $f$ is constant.
        \item $R\cong\K[x]$, where $\K=R^\times\cup\{0\}$ is a field, and there is a strictly increasing function $\phi:\Z_0^+\to\Z_0^+$ such that $f=\phi\circ\deg$.
    \end{enumerate}
\end{corollary}

A consequence of \Cref{uniquely_euclidean_function_characterization} is that every uniquely Euclidean function $f$ on $\K[x]$ must be constant among polynomials of the same degree. If we only assume $f$ is ultra-Euclidean, the result fails, as the following example shows:

\begin{example}
    Suppose $R=\F_4[x]$, where $\F_4$ is the field with four elements. Define $f:\F_4[x]\to\Z_0^+$ by $f(1)=0$ and $f(p)=\deg(p)+1$ for $p\not\equiv 1$. Then $f$ is ultra-Euclidean on $\F_4[x]$, but $f$ is not constant among polynomials of degree $0$ (since its value is $0$ at $1$ and $1$ at other nonzero constants).
\end{example}

\section{Further investigation}

At this point, a natural question to ask is the following:

\begin{conjecture}[Conjecture]
    If $f$ is ultra-Euclidean on $R$, so is $\ftil$.
\end{conjecture}

We were not able to prove or disprove this conjecture. If true, this would allow us to further strengthen \Cref{uniquely_euclidean_implies_K_or_Kx}:

\begin{conjecture}[Corollary]
    If $R$ has an ultra-Euclidean function, then either $R$ is a field or $R\cong\K[x]$ for some field $\K$.
\end{conjecture}

Any counterexample to the conjecture must be a Euclidean domain $R$ with a function $f:R\exc{0}\to\Z_0^+$ such that $f$ is ultra-Euclidean but \emph{not} strongly Euclidean, while $\ftil$ is strongly Euclidean but \emph{not} ultra-Euclidean.

\subsection*{Acknowledgments}

I would like to thank Thomas Preu for his corrections and suggestions.

\phantomsection 
\addcontentsline{toc}{section}{References} 
\printbibliography 

\end{document}